\documentclass[11pt]{article}
\usepackage{amsmath,amssymb,amsthm}
\usepackage[margin=2cm]{geometry}
\usepackage{url}
\newtheorem{theorem}{Theorem}
\newtheorem{lemma}{Lemma}
\newtheorem{remark}{Remark}

\newcommand{\inj}{\operatorname{inj}}
\newcommand{\Ric}{\operatorname{Ric}}
\newcommand{\Rm}{\operatorname{Rm}}
\newcommand{\secop}{\operatorname{sec}}
\newcommand{\conjrad}{\operatorname{conj\,rad}}
\newcommand{\vol}{\operatorname{vol}}
\newcommand{\dist}{\operatorname{dist}}

\title{Bi-Lipschitz Smoothing under Ricci and Injectivity Bounds}
\author{Maja Gw\'o\'zd\'z\\University of Zurich \& ETH Z\"{u}rich, Zurich, Switzerland\\\texttt{mgwozdz@ethz.ch}}
\date{}

\begin{document}
\maketitle

\begin{abstract}
We prove that a complete Riemannian manifold with a positive
uniform lower bound on injectivity radius and a positive uniform lower bound on
Ricci curvature admits an $L^\infty$-close (bi-Lipschitz) smooth metric with
two-sided Ricci curvature bounds and a uniform positive lower bound on
injectivity radius. This answers Question~2 in the
Morgan--Pansu list of open problems from the conference \emph{Modern Trends in
Differential Geometry} (S\~ao Paulo, 2018), proposed by L.~Bandara.
In the proof, we rely on controlled smoothing \cite{PWY} with Croke's universal
local volume lower bound \cite{Croke} and the Cheeger--Gromov--Taylor
injectivity radius estimate \cite{AM}.
\end{abstract}

\section{Introduction}

In this note, we work with smooth manifolds without boundary. Let $(M^n,g)$ be a complete Riemannian manifold with $\inj(M,g)\ge \ell>0$ and
$\Ric_g\ge k\,g$ for some $k>0$. Our goal is to form
a bi-Lipschitz nearby smooth metric $h$ with a two-sided Ricci bound
and a uniform positive injectivity radius bound. More precisely, we solve the following problem.

\medskip\noindent
\textbf{Question~2 (due to Bandara \cite{MP}).}
\textit{For every \(\ell, k > 0\), there exist \(C, L, K > 0\) with the following effect. Let \((M, g)\) be a complete Riemannian manifold with injectivity radius \(inj(M, g) \ge \ell\) and Ricci curvature \(Ric(g) \ge k\). Then there exists a metric \(h\) on \(M\) with \(inj(M, h) \ge L\) and \(|Ric(h)| \le K\) such that}
\[
\frac{1}{C} g \le h \le C g.
\]

\noindent The proof below is built on known results, namely,
controlled smoothing \cite{PWY}, Croke's universal local volume lower bound \cite{Croke}, and the Cheeger--Gromov--Taylor injectivity radius estimate \cite{AM}.

\section{Main result}

\begin{theorem}[Smoothing under $\inj$ and $\Ric$ lower bounds]\label{thm:main}
We restate Question~2 in more detail. Let us fix an integer $n\ge 2$. For every $\ell>0$ and $k>0$, there exist constants
$C,L,K>0$ (that depend only on $n,\ell,k$) with the following property.

\noindent Let $(M^n,g)$ be a complete Riemannian manifold that satisfies:
\[
\inj(M,g)\ge \ell,
\qquad
\Ric_g \ge k\, g \quad\text{(as quadratic forms).}
\]
There exists a smooth complete Riemannian metric $h$ on $M$ with:
\[
\frac1C\,g \le h \le C\,g,
\qquad
\inj(M,h)\ge L,
\qquad
|\Ric_h|_h\le K .
\]
Here $\inj(M,g):=\inf_{p\in M}\inj_g(p)$ and $|\cdot|_h$ denotes the pointwise norm
computed using $h$.
\end{theorem}

\begin{lemma}[Metric comparison]\label{lem:bilip}
Let $g,h$ be Riemannian metrics on $M$ and let $C\ge 1$.
If $C^{-1}g \le h \le Cg$, then for all $x,y\in M$, we have:
\[
C^{-1/2} \dist_g(x,y) \le \dist_h(x,y)\le C^{1/2} \dist_g(x,y),
\]
and pointwise on $M$,
\[
C^{-n/2}\,d\vol_g \le d\vol_h \le C^{n/2}\,d\vol_g.
\]
In particular, notice that if $(M,g)$ is complete, then $(M,h)$ is complete.
\end{lemma}

\begin{proof}
For an arbitrary tangent vector $v$, it holds that:
\[
|v|_h^2=h(v,v)\le Cg(v,v)=C|v|_g^2,\qquad
|v|_h^2\ge C^{-1}|v|_g^2,
\]
so $|v|_h\le C^{1/2}|v|_g$ and $|v|_h\ge C^{-1/2}|v|_g$. We now integrate along the curves and obtain the required distance bounds. For the volume forms, we choose a $g$-orthonormal basis so that $g$ is represented by $I$
and $h$ by a positive definite matrix $A$. It follows that $C^{-1}I\le A\le CI$, so we infer that $\det A\in[C^{-n},C^n]$ and $d\vol_h = \sqrt{\det A}\,d\vol_g$. We obtain completeness directly from the distance comparison. Observe that an $h$-Cauchy sequence is $g$-Cauchy, so it converges in $(M,g)$. In particular, it also converges in $(M,h)$ by the reverse comparison, which concludes the proof.
\end{proof}

\begin{lemma}[Constant rescaling]\label{lem:scale}
Let $\lambda>0$ and $\hat g := \lambda g$. It follows that:
\[
\inj(M,\hat g)=\sqrt{\lambda}\,\inj(M,g).
\]
Moreover, note that the Levi--Civita connections of $g$ and $\hat g$ are the same. In particular, we have:
\[
\Ric_{\hat g}=\Ric_g \quad\text{as $(0,2)$-tensors}.
\]
As a result, for an arbitrary $(0,2)$-tensor $T$, the following holds:
\[
|T|_{\hat g}=\lambda^{-1}|T|_g.
\]
\end{lemma}

\begin{proof}
Lengths and distances scale by $\sqrt{\lambda}$ under $g\mapsto \lambda g$, and so does the injectivity radius. For constant rescaling, we have:
\[
\begin{aligned}
\Gamma^k_{ij}(\hat g)
&=\frac12 \hat g^{k\ell}(\partial_i \hat g_{j\ell}+\partial_j\hat g_{i\ell}-\partial_\ell \hat g_{ij})\\
&=\frac12 (\lambda^{-1}g^{k\ell})(\lambda\partial_i g_{j\ell}+\lambda\partial_j g_{i\ell}-\lambda\partial_\ell g_{ij})\\
&=\Gamma^k_{ij}(g),
\end{aligned}
\]
so the connections are the same. Notice that the $(1,3)$ curvature tensor is unchanged, and so is $\Ric$ as a $(0,2)$-tensor. Finally, we observe that $|T|_{\lambda g}^2=(\lambda g)^{ia}(\lambda g)^{jb}T_{ij}T_{ab}=\lambda^{-2}|T|_g^2$.
\end{proof}

\begin{lemma}[Morrey--Sobolev on a unit ball]\label{lem:morrey-sobolev}
Let $p>n$ and set $\alpha:=1-\frac{n}{p}\in(0,1)$. There exists a constant
$C_{\mathrm{M}}=C_{\mathrm{M}}(n,p)>0$ such that for every
$u\in W^{1,p}(B(0,1))$, there exists a representative (denoted by $u$) with
$u\in C^{0,\alpha}(B(0,1))$ and
\[
[u]_{C^{0,\alpha}(B(0,1))}\le C_{\mathrm{M}}\|\nabla u\|_{L^p(B(0,1))}.
\]
\end{lemma}

\begin{proof}
See, for instance, \cite[p.~266]{EvansPDE}.
\end{proof}

\begin{lemma}[Morrey estimate]\label{lem:morrey}
Let $p>n$ and set $\alpha=1-\frac{n}{p}$. Let $C_{\mathrm{M}}=C_{\mathrm{M}}(n,p)$ be the constant from
Lemma~\ref{lem:morrey-sobolev}. It follows that for every $r>0$ and every $u\in W^{1,p}(B(0,r))$,
\[
r^{\alpha}\,[u]_{C^{0,\alpha}(B(0,r))}\ \le\ C_{\mathrm{M}}\, r^{1-\frac{n}{p}}\|\nabla u\|_{L^p(B(0,r))}.
\]
\end{lemma}

\begin{proof}
The strategy is to scale to $B(0,1)$. Let us define $u_r(x):=u(rx)$. It follows that:
\begin{align*}
\|\nabla u_r\|_{L^p(B(0,1))} &= r^{1-\frac{n}{p}}\|\nabla u\|_{L^p(B(0,r))},\\
[u_r]_{C^{0,\alpha}(B(0,1))} &= r^{\alpha}[u]_{C^{0,\alpha}(B(0,r))}.
\end{align*}
To conclude the proof, it suffices to apply Lemma~\ref{lem:morrey-sobolev} to $u_r$.
\end{proof}

\begin{lemma}\label{lem:curv-coeff}
Let $U\subset\mathbb{R}^n$ be open and let $m$ be a $C^2$ Riemannian metric on $U$ with coefficients $(m_{ij})$ in the standard coordinates. Let us assume that for some constants $c\ge 1$ and $A,B\ge 0$, it holds that:
\[
c^{-1}\delta \le (m_{ij})\le c\,\delta,\qquad
\sup_U|\partial m_{ij}|\le A,\qquad
\sup_U|\partial^2 m_{ij}|\le B.
\]
We then know that there exists a constant $C_{\mathrm{curv}}=C_{\mathrm{curv}}(n,c)$ such that on $U$,
\[
|\Rm_m|_m \le C_{\mathrm{curv}}\,(B+A^2).
\]
\end{lemma}

\begin{proof}
Let $(m^{ij})$ be the inverse matrix. By ellipticity, we obtain $\sup_U|m^{ij}|\le C_0(n,c)$. Furthermore, notice that the Christoffel symbols satisfy $|\Gamma|\le C_1(n,c)A$.
We now differentiate $m^{ip}m_{pj}=\delta^i_j$ and get $|\partial m^{ij}|\le C_2(n,c)A$,
so $|\partial\Gamma|\le C_3(n,c)(A^2+B)$.
The curvature components are precisely:
\[
R^k{}_{ij\ell}
=
\partial_i\Gamma^k_{j\ell}-\partial_j\Gamma^k_{i\ell}
+\Gamma^k_{ip}\Gamma^p_{j\ell}-\Gamma^k_{jp}\Gamma^p_{i\ell},
\]
so $\sup|R^k{}_{ij\ell}|\le C_4(n,c)(A^2+B)$. By ellipticity, we can then transform this to the invariant norm, as required.
\end{proof}

\begin{lemma}[Croke]\label{lem:croke}
There exists a constant $v_n>0$ that depends only on $n$ such that the following holds. If $(N^n,m)$ is a compact Riemannian manifold without boundary, then for every $p\in N$ and every $0<r\le \inj(N,m)/2$, it holds that:
\[
\vol_m\bigl(B_m(p,r)\bigr)\ge v_n\,r^n.
\]
\end{lemma}

\begin{proof}
The result is due to \cite{Croke}.
\end{proof}

\begin{lemma}[Cheeger--Gromov--Taylor]\label{lem:CGT}
Let $(N^n,m)$ be a connected complete Riemannian manifold with sectional curvature bounds
$\lambda\le \secop_m\le \Lambda$. Let $p\in N$ and let $r>0$ with $r<\pi/(4\sqrt{\Lambda})$ if $\Lambda>0$.
It then follows that:
\[
\inj_m(p)\ \ge\ r\,
\frac{\vol_m(B_m(p,r))}{\vol_m(B_m(p,r)) + V^n_\lambda(2r)},
\]
where $V^n_\lambda(\rho)$ denotes the volume of the radius-$\rho$ ball in the simply connected $n$-dimensional space form of constant sectional curvature $\lambda$.
\end{lemma}

\begin{proof}
This is the first inequality in \cite[Theorem~3.7]{AM}.
\end{proof}

\noindent We now have all the auxiliary lemmas in place. Let us now consider the proof of the main theorem.

\begin{proof}[Proof of Theorem~\ref{thm:main}]
We understand all inequalities between metrics in the sense of quadratic forms. For a $(0,2)$-tensor $T$, we compute the pointwise norm $|T|_m$ with respect to the metric $m$. If $M$ is not connected, we apply the argument below to each connected component. Given the fact that the hypotheses and the constants depend only on $(n,\ell,k)$, the constants we obtain are uniform. Therefore, we assume that $M$ is connected.

\medskip\noindent
We first scale so that $\inj(M,\bar g)\ge 1$. Note that compactness holds.
Given that $\Ric_g\ge k g$ with $k>0$, we set $H:=k/(n-1)>0$ so that
$\Ric_g\ge (n-1)H\,g$. By Bonnet--Myers, we know that $(M,g)$ has finite diameter, and since it is complete, the Hopf--Rinow theorem implies that $(M,g)$ is compact (in equivalent terms, it is closed).

We set:
\[
\bar g := \ell^{-2} g.
\]
By Lemma~\ref{lem:scale}, we obtain:
\begin{equation}\label{eq:injbar}
\inj(M,\bar g)=\ell^{-1}\inj(M,g)\ge 1.
\end{equation}
It also holds that $\Ric_{\bar g}=\Ric_g$ as $(0,2)$-tensors, and so
\[
\Ric_{\bar g}\ge k\,g = k\ell^2\,\bar g,
\]
so, in particular, $\Ric_{\bar g}\ge 0$ holds.

It suffices to prove the theorem for $(M,\bar g)$ with $\inj(M,\bar g)\ge 1$. Note that if
$\bar h$ satisfies
\[
C^{-1}\bar g \le \bar h \le C\bar g,\qquad
\inj(M,\bar h)\ge \bar L,\qquad
|\Ric_{\bar h}|_{\bar h}\le \bar K,
\]
then setting $h:=\ell^2\bar h$ and applying Lemma~\ref{lem:scale} gives
\[
\begin{aligned}
C^{-1} g &\le h \le C g,\\
\inj(M,h)&=\ell\,\inj(M,\bar h)\ge \ell\bar L,\\
|\Ric_h|_h&=\ell^{-2}|\Ric_{\bar h}|_{\bar h}\le \ell^{-2}\bar K.
\end{aligned}
\]

\medskip\noindent
We will now consider a conjugate-radius lower bound.
To this end, for $p\in M$, let $\conjrad_{\bar g}(p)$ denote the conjugate radius at $p$.
Observe that if $\exp_p$ is a diffeomorphism on the open $\bar g$-ball of radius $r$ in $T_pM$, then $\exp_p$ has no critical points there. This implies that $r\le \conjrad_{\bar g}(p)$. We then obtain:
\[
\inj_{\bar g}(p)\le \conjrad_{\bar g}(p)\qquad\forall p\in M.
\]
We apply \eqref{eq:injbar} and get
\begin{equation}\label{eq:conjbar}
\conjrad_{\bar g}(p)\ge 1\qquad\forall p\in M.
\end{equation}

\medskip\noindent
We now work towards establishing weak harmonic $C^{0,\alpha}$ control.
Let us first fix
\[
p_0:=2n>n,\qquad \alpha:=1-\frac{n}{p_0}=\frac12.
\]
Note that by \eqref{eq:conjbar} and $\Ric_{\bar g}\ge 0$, we may apply the results from \cite{PWY} (if we take $H=0$ and $\rho=1$). More precisely, there exists a function
$Q_1:(0,1]\to(0,\infty)$ with $\lim_{r\downarrow 0}Q_1(r)=0$ such that for every $r\in(0,1]$, the following is true:
\begin{equation}\label{eq:weak-harmonic-L1p}
\|(M,\bar g)\|_{W;hL^{1,p_0},r}\ \le\ Q_1(r),
\end{equation}
where $\|\cdot\|_{W;hL^{1,p_0},r}$ denotes the \emph{weak harmonic} $L^{1,p_0}$-norm on scale $r$ in the sense of \cite{PWY}. We stress that this is the weak \emph{harmonic}
$L^{1,p_0}$-norm (not just the weak $L^{1,p_0}$-norm). The results from \cite{PWY} provide the required harmonic coordinate charts. We also absorb absolute factors into $Q$. Finally, if we replace $Q_1$ by its monotone envelope, we may assume $Q_1$ is nondecreasing.

Let us fix $r\in(0,1]$.
We now analyze the definition of $\|(M,\bar g)\|_{W;hL^{1,p_0},r}$ \cite{PWY} in more detail. \eqref{eq:weak-harmonic-L1p} means that for each $p\in M$, there exists a local diffeomorphism
$\varphi:B(0,r)\to U\subset M$ with $p\in U$ with the property that the pullback metric coefficients
$(\bar g_{ij})=(\varphi^*\bar g)_{ij}$ satisfy:
\begin{equation}\label{eq:L1p-chart-bounds}
e^{-Q_1(r)}\delta \le (\bar g_{ij}) \le e^{Q_1(r)}\delta,
\qquad
r^{1-\frac{n}{p_0}}\|\partial \bar g_{ij}\|_{L^{p_0}(B(0,r))}\le Q_1(r),
\end{equation}
in the sense of \cite{PWY}. We keep the normalization $e^{\pm Q_1(r)}$.

In particular, the quadratic-form bounds imply that
\[
|\bar g_{ij}|\le \sqrt{\bar g_{ii}\bar g_{jj}}\le e^{Q_1(r)} \qquad\text{on }B(0,r).
\]

We now apply Lemma~\ref{lem:morrey} componentwise to $\bar g_{ij}$ and obtain
\[
r^{\alpha}\,[\bar g_{ij}]_{C^{0,\alpha}(B(0,r))}\le C_{\mathrm{M}}(n,p_0)\,Q_1(r).
\]
It follows that
\[
\begin{aligned}
r^{\alpha}\|\bar g_{ij}\|_{C^{0,\alpha}(B(0,r))}
&\le r^{\alpha}\|\bar g_{ij}\|_{L^\infty(B(0,r))}\\
&\qquad + r^{\alpha}[\bar g_{ij}]_{C^{0,\alpha}(B(0,r))}\\
&\le r^{\alpha}e^{Q_1(r)} + C_{\mathrm{M}}Q_1(r).
\end{aligned}
\]
We now define, for $0<r\le 1$,
\[
Q_0(r):= \max\Bigl\{\,Q_1(r),\ r^{\alpha}e^{Q_1(r)} + C_{\mathrm{M}}Q_1(r)\Bigr\}.
\]
We infer that $Q_0$ is nondecreasing and $\lim_{r\downarrow 0}Q_0(r)=0$. We extend $Q_0$ to
$(0,\infty)$ by setting $Q_0(r):=Q_0(1)$ for $r\ge 1$. Note that this keeps $Q_0$ nondecreasing. Given that $Q_0(r)\ge Q_1(r)$ holds, we may apply the same harmonic charts as in
\eqref{eq:weak-harmonic-L1p}. The charts are unchanged,
so they remain harmonic. Moreover, the ball-containment condition in the definition
of the weak norm is preserved, because if we replace $Q_1$ by $Q_0\ge Q_1$, we merely shrink the
required ball. It follows that
\begin{equation}\label{eq:weak-harmonic-C0a}
\|(M,\bar g)\|_{W;hC^{0,\alpha},r}\ \le\ Q_0(r)\qquad\forall r\in(0,1].
\end{equation}
So we have $(M,\bar g)\in M(n,\alpha,Q_0)$ in the sense of \cite[Definition~1]{PWY}.

\medskip\noindent
We now apply the smoothing theorem of Petersen--Wei--Ye \cite[Theorem~1.1]{PWY}
to $(M,\bar g)\in M(n,\alpha,Q_0)$. To this end, fix $\varepsilon=1$ and $m=2$.
We observe that Theorem~1.1 is uniform for all $0<r\le 1$, and the parameter $r$ occurs only in the weak norms
$\|\cdot\|_{W;C^{k,\alpha},r}$. We use the weak $C^{2,\alpha}$ bound at scale $r=1$,
so we take $Q_0(1)$ in the constants. Theorem~1.1 is valid for an arbitrary integer $m\ge 2$. Most importantly, its smoothing construction
\cite{PWY} forms a smooth metric with uniform $C^{m,\alpha}$ bounds.
In the argument below, we only use the $C^{2,\alpha}$ bounds (so we take $m=2$).
We arrive at a Riemannian metric $\bar h$ on $M$ with the property that
\begin{equation}\label{eq:bilip}
e^{-1}\bar g \le \bar h \le e\,\bar g,
\end{equation}
and such that its weak $C^{2,\alpha}$-norm on scale $1$ is bounded:
\begin{equation}\label{eq:weakC2}
\|(M,\bar h)\|_{W;C^{2,\alpha},1}\le Q_2,
\end{equation}
where $Q_2$ depends only on $n$, $\alpha$, $\varepsilon$ (here $1$), $m=2$, and $Q_0(1)$.

We now consider \eqref{eq:weakC2}, that is, for each $p\in M$, there exist a local diffeomorphism
$\varphi:B(0,1)\subset\mathbb{R}^n\to U\subset M$ and a point $\tilde p\in B(0,1)$ with
$\varphi(\tilde p)=p$ such that, if we write $(\bar h_{ij})=(\varphi^\ast\bar h)_{ij}$ on $B(0,1)$, we get:
\begin{align}
e^{-Q_2}\delta &\le (\bar h_{ij}) \le e^{Q_2}\delta, \label{eq:elliptic}\\
\max_{|\beta|\le 2}\sup_{B(0,1)}|\partial^\beta \bar h_{ij}|
\;+\;
[\partial^2 \bar h_{ij}]_{C^\alpha(B(0,1))}
&\le Q_2 . \label{eq:C2bounds}
\end{align}

\medskip\noindent
We next work towards determining the uniform curvature and Ricci bounds for $\bar h$.
We fix $p\in M$ and choose $\varphi$ as in the smoothing construction above. We also apply Lemma~\ref{lem:curv-coeff} to
$m=\varphi^\ast \bar h$ on $B(0,1)$, using \eqref{eq:elliptic}--\eqref{eq:C2bounds}. It follows that there exists
$\Lambda=\Lambda(n,Q_2)$ such that
\[
|\Rm_{\bar h}|_{\bar h}(p)\le \Lambda.
\]
Given that $p$ was arbitrary, it follows that
\begin{equation}\label{eq:Rm-bound}
|\Rm_{\bar h}|_{\bar h}\le \Lambda \qquad\text{on }M.
\end{equation}

Observe that since $\Ric$ is a contraction of $\Rm$, there exists $c_{\mathrm{Ric}}=c_{\mathrm{Ric}}(n)>0$
such that
\begin{equation}\label{eq:Ric-bound}
|\Ric_{\bar h}|_{\bar h}\le c_{\mathrm{Ric}}\,|\Rm_{\bar h}|_{\bar h}\le c_{\mathrm{Ric}}\,\Lambda \;=:\; \bar K.
\end{equation}
We also know that $|\sec_{\bar h}|\le |\Rm_{\bar h}|_{\bar h}$, and so
\begin{equation}\label{eq:sec-bounds}
-\Lambda \le \secop_{\bar h}\le \Lambda \qquad\text{on }M.
\end{equation}

\medskip\noindent
We now obtain a uniform lower volume bound for $\bar h$-balls.
We apply Lemma~\ref{lem:croke} to $(M,\bar g)$ and use \eqref{eq:injbar} to get, for all $p\in M$ and
all $0<r\le 1/2$,
\begin{equation}\label{eq:volgbar-lower}
\vol_{\bar g}\bigl(B_{\bar g}(p,r)\bigr)\ge v_n\,r^n.
\end{equation}

We also apply the bi-Lipschitz comparison \eqref{eq:bilip}. By Lemma~\ref{lem:bilip} with $C=e$, it follows that
$d_{\bar h}\le e^{1/2}d_{\bar g}$, so
\[
B_{\bar g}\bigl(p,e^{-1/2}r\bigr)\subset B_{\bar h}(p,r).
\]
We also have $d\vol_{\bar h}\ge e^{-n/2}d\vol_{\bar g}$. Therefore, for $0<r\le 1/2$,
\begin{equation}\label{eq:volhbar-lower}
\vol_{\bar h}\bigl(B_{\bar h}(p,r)\bigr)\ge v_n e^{-n}\,r^n.
\end{equation}

\medskip\noindent
It remains to estimate the injectivity radius of $\bar h$ using Cheeger--Gromov--Taylor.
We first apply Lemma~\ref{lem:CGT} to $(M,\bar h)$, use \eqref{eq:sec-bounds}, and take $\lambda=-\Lambda$.
Let us define:
\[
r_* :=
\begin{cases}
\min\bigl\{\,\frac12,\ \frac{\pi}{8\sqrt{\Lambda}}\,\bigr\}, & \Lambda>0,\\[4pt]
\frac12, & \Lambda=0.
\end{cases}
\]
It follows that $0<r_*\le 1/2$, and if $\Lambda>0$, we have $r_*<\pi/(4\sqrt{\Lambda})$.

By \eqref{eq:volhbar-lower}, we get:
\[
\vol_{\bar h}(B_{\bar h}(p,r_*))\ge v_n e^{-n} r_*^n.
\]
We infer that Lemma~\ref{lem:CGT} gives, for every $p\in M$,
\[
\inj_{\bar h}(p)\ \ge\ r_*\,
\frac{v_n e^{-n} r_*^n}{v_n e^{-n} r_*^n + V^n_{-\Lambda}(2r_*)}
\ =:\ \bar L.
\]
In particular, it is true that
\begin{equation}\label{eq:injhbar}
\inj(M,\bar h)\ge \bar L>0.
\end{equation}

\medskip\noindent
Finally, let us now return to the original scale.
We set $h:=\ell^2\bar h$. By \eqref{eq:bilip}, we have
\[
e^{-1}g \le h \le e\,g,
\]
so we may take $C:=e$. By \eqref{eq:injhbar} and Lemma~\ref{lem:scale},
\[
\inj(M,h)=\ell\,\inj(M,\bar h)\ge \ell\,\bar L,
\]
so we take $L:=\ell\,\bar L$. Finally, by \eqref{eq:Ric-bound} and Lemma~\ref{lem:scale},
\[
|\Ric_h|_h=\ell^{-2}|\Ric_{\bar h}|_{\bar h}\le \ell^{-2}\bar K,
\]
so we set $K:=\ell^{-2}\bar K$. This concludes the proof.
\end{proof}

\begin{remark}
Note that we used the Bonnet--Myers theorem only to obtain compactness, which, in turn, was required to apply Croke's result. In the rest of the proof, we used $\Ric_{\bar g}\ge 0$ with the lower injectivity
radius bound. If we were to assume compactness separately and still assume $\Ric_{\bar g}\ge 0$ and $\inj_{\bar g}\ge 1$ (in other terms, $\Ric_g\ge 0$ and
$\inj_g\ge \ell$ before scaling), then we could take the constants to depend only on $n$ and $\ell$.
\end{remark}

\section*{Statements and Declarations}
\subsection*{Competing Interests}
The author declares that there are no competing interests.

\end{document}